\documentclass{amsart}

\usepackage{mathrsfs}
\usepackage[dvips]{graphicx}
\usepackage{amsmath}
\usepackage[latin1]{inputenc}
\usepackage{amsfonts}
\usepackage{amssymb}
\usepackage{graphicx,epsfig}
\usepackage{amsthm}

\usepackage{amscd}
\usepackage{dsfont}
\usepackage[all,dvips]{xy}
\usepackage{fancyhdr}
\usepackage[T1]{fontenc}

\input cyracc.def

\newtheorem{theorem}{Theorem}[section]
\newtheorem{lemma}[theorem]{Lemma}
\newtheorem{proposition}[theorem]{Proposition}

\theoremstyle{definition}
\newtheorem{definition}[theorem]{Definition}
\newtheorem{example}[theorem]{Example}

\theoremstyle{remark}
\newtheorem{remark}[theorem]{Remark}

\allowdisplaybreaks

\numberwithin{equation}{section}

\begin{document}

\title{Explicit results concerning quantum quasi-shuffle algebras and their applications}

\author{Run-Qiang Jian}
\address{School of Computer Science, Dongguan University
of Technology, 1, Daxue Road, Songshan Lake, 523808, Dongguan, P.
R. China}
\email{jian.math@gmail.com}
\thanks{}


\date{}

\dedicatory{}


\begin{abstract}
Using the concept of mixable shuffles, we formulate explicitly the
quantum quasi-shuffle product. We also provide a desirable
description of the subalgebra generated by the set of primitive
elements of the quantum quasi-shuffle bialgebra. A braided
coalgebra structure which is dual to the quantum quasi-shuffle in
some sense is constructed as well. We use quantum quasi-shuffle
algebras to provide examples of Rota-Baxter algebras and
tridendriform algebras.
\end{abstract}

\maketitle
\section{Introduction}
In \cite{Ree}, Ree introduced the shuffle algebra which has been
studied extensively during the last fifty years. The shuffle
product is carried out on the tensor space $T(V)$ of a vector
space $V$ by using the shuffle rule. Its natural generalization is
the quasi-shuffle product where $V$ is moreover an associative
algebra and the new product on $T(V)$ involves both of the shuffle
product and the multiplication of $V$. Quasi-shuffle algebras
first arose in the work of Newman and Radford \cite{NR} for the
study of cofree irreducible Hopf algebras built on associative
algebras. Later, they were discussed by many other mathematicians
for various motivations, such as multiple zeta values (\cite{Hof1}
and \cite{IKZ}), Rota-Baxter algebras (\cite{GK},\cite{G} and
\cite{EG}), and commutative tridendriform algebras (\cite{Lod}).

For both of physical and mathematical considerations, people wants
to deform or quantize some important algebra structures. The most
famous example is absolutely the quantum group introduced by
Drinfeld \cite{D} and Jimbo \cite{J}. To our surprise, there is an
implicit but significant connection between quantum groups and
shuffle algebras. Rosso \cite{Ro} constructed the quantization of
shuffle algebras. This is a new kind of quantized algebras and
leads to an intrinsic understanding of the quantum group. Since
shuffle algebras are special quasi-shuffle algebras, and the
importance of the later ones, people would expect to find out what
the quantization of quasi-shuffle algebras is and whether it can
bring us some useful information. Hoffman's q-deformation of
quasi-shuffles (\cite{Hof2}) is such an attempt. His idea is to
deform the multiplication formula according to a special case of
Rosso's quantum shuffles. After replacing the shuffle part by its
quantized version in the formula of quasi-shuffles, Hoffman tried
to deform the mixed term by multiplying a power of q, and found
that there is only one way making the new product to be
associative. This construction is more or less experiential as he
said. The general construction of quantized quasi-shuffles is due
to Rosso (\cite{Ro2}) in the spirit of his quantum shuffles. We
describe Rosso's idea as follows. Let $M$ be
 a Hopf bimodule over a Hopf algebra $H$. In addition, if $M$ is an algebra and the multiplication is
compatible with the braiding coming from the Hopf bimodule
structure in some sense, then one can construct a new algebra
structure on the cotensor coalgebra $T^c_H(M)$ by using its
universal property. In general, given a braided algebra
$(A,m,\sigma)$, one can construct an analogue of the quasi-shuffle
algebra in the braided category, where the action of the usual
flip is replaced by that of the braiding. The resulting algebra is
called a quantum quasi-shuffle algebra. In particular, Hoffman's
q-deformation of quasi-shuffle products is a special case of
Rosso's quantum quasi-shuffle. In \cite{JR}, the construction of
quantum quasi-shuffles appears, as a special braided cofree Hopf
algebra, in the framework of quantum multi-brace algebras.

 Some interesting properties of the quantum
quasi-shuffle algebras have been studied in \cite{JRZ}, including
the commutativity, universal property, and etc. This paper
continues the trip. We establish some explicit results concerning
this new subject. We start by reformulating the product.
Originally, the quantum quasi-shuffle algebra is constructed by
using the universal property of connected coalgebras (\cite{JR}).
Later, it is defined by an inductive formula (\cite{JRZ}). But
neither of these two constructions can provide an explicit
formula. To know more about this new subject, a more clear form of
the multiplication formula is definitely helpful. Here, we use the
notion of mixable shuffles introduced in \cite{GK} to establish a
complete description of the quantum quasi-shuffle product.
Contrast to the quantum symmetric algebra, we describe the
subalgebra of the quantum quasi-shuffle bialgebra generated by the
primitive elements. On the other hand, for the reason that the
universal property of connected coalgebras is not so familiar by
non-algebraists, we use the universal property of tensor algebras
to construct a braided coalgebra structure on $T(C)$ for a braided
coalgebra $C$, and show that its dual is the quantum quasi-shuffle
algebra. This enables one to study the quantum quasi-shuffle
algebra through its dual. Finally, as applications, we provide
examples of Rota-Baxter algebras and tridendriform algebras by
using quantum quasi-shuffles. Rota-Baxter algebras and
tridendriform algebras are important subjects in mathematics and
physics. So our constructions not only enlarge their families, but
also demonstrate the value of the quantum quasi-shuffles.

This paper is organized as follows. In Section 2, several concrete
examples of braided algebras are provided. In Section 3, we recall
the construction of quantum quasi-shuffle algebras and established
explicit formulas for the new product and the subalgebra generated
by the primitive elements. In Section 4, we construct the dual
coalgebra of the quantum quasi-shuffle algebra. Finally, in
Section 5, it contains some applications of the quantum
quasi-shuffle algebra involving examples of Rota-Baxter algebras
and tridendriform algebras.

\noindent\textbf{Notation.} In this paper, we denote by
$\mathbb{K}$ a ground field of characteristic 0. All the objects
we discuss are defined over $\mathbb{K}$. For a vector space $V$,
we denote by $\otimes$ the tensor product within $T(V)$, and by
$\underline{\otimes}$ the one between $T(V)$ and $T(V)$.

We denote by $\mathfrak{S}_{n}$ the symmetric group of the set
$\{1,2,\ldots,n\}$ and by $s_{i}$, $1\leq i\leq n-1$, the standard
generators of $\mathfrak{S}_{n}$ permuting $i$ and $i+1$. For
fixed $k,n\in \mathbb{N}$, we define the shift map
$\mathrm{shift}_k:\mathfrak{S}_{n}\rightarrow\mathfrak{S}_{n+k}$
by $\mathrm{shift}_k(s_i)=s_{i+k}$ for any $1\leq i\leq n-1$. For
the reason of intuition and the simplicity of notation, we denote
$1_{\mathfrak{S}_{k}}\times w=\mathrm{shift}_k(w)$ for any $w\in
\mathfrak{S}_{n}$. The notations $w\times 1_{\mathfrak{S}_{k}}$,
$1_{\mathfrak{S}_{k}}\times w\times 1_{\mathfrak{S}_{l}}$ and
others are understood similarly.

A braiding $\sigma$ on a vector space $V$ is an invertible linear
map in $\mathrm{End}(V\otimes V)$ satisfying the quantum
Yang-Baxter equation on $V^{\otimes 3}$: $$(\sigma\otimes
\mathrm{id}_{V})(\mathrm{id}_{V}\otimes \sigma)(\sigma\otimes
\mathrm{id}_{V})=(\mathrm{id}_{V}\otimes \sigma)(\sigma\otimes
\mathrm{id}_{V})(\mathrm{id}_{V}\otimes \sigma).$$ A braided
vector space $(V,\sigma)$ is a vector space $V$ equipped with a
braiding $\sigma$. For any $n\in \mathbb{N}$ and $1\leq i\leq
n-1$, we denote by $\sigma_i$ the operator $\mathrm{id}_V^{\otimes
i-1}\otimes \sigma\otimes \mathrm{id}_V^{\otimes n-i-1}\in
\mathrm{End}(V^{\otimes n})$. For any $w\in \mathfrak{S}_{n}$, we
denote by $T^\sigma_w$ the corresponding lift of $w$ in the braid
group $B_n$, defined as follows: if $w=s_{i_1}\cdots s_{i_l}$ is
any reduced expression of $w$, then $T^\sigma_w=\sigma_{i_1}\cdots
\sigma_{i_l}$. This definition is well-defined (see, e.g., Theorem
4.12 in \cite{KT}).

We define $\beta:T(V)\underline{\otimes} T(V)\rightarrow
T(V)\underline{\otimes} T(V)$ by requiring that the restriction of
$\beta$ on $V^{\otimes i}\underline{\otimes} V^{\otimes j}$,
denoted by $\beta_{ij}$, is $T^\sigma_{\chi_{ij}}$ , where
\[\chi_{ij}=\left(\begin{array}{cccccccc}
1&2&\cdots&i&i+1&i+2&\cdots & i+j\\
j+1&j+2&\cdots&j+i&1& 2 &\cdots & j
\end{array}\right)\in \mathfrak{S}_{i+j},\] for any $i,j\geq 1$. For convenience, we denote by
$\beta_{0i}$ and $\beta_{i0}$ the usual flip.

\section{Braided algebras}
We start by recalling the notion of braided algebras which is the
correspondent object of associative algebras in braided
categories. In the following, all algebras are assumed to be
associative but not necessarily unital.
\begin{definition}Let $A=(A,m)$ be an algebra with product $m$, and $\sigma$ be a braiding on $A$. We call the triple $(A,m,\sigma)$ a \emph{braided algebra} if it satisfies the following conditions:
\[\left\{
\begin{array}{lll}
(\mathrm{id}_A\otimes m)\sigma_1\sigma_2&=&\sigma( m\otimes
\mathrm{id}_A),\\[3pt] ( m\otimes
\mathrm{id}_A)\sigma_2\sigma_1&=&\sigma(\mathrm{id}_A\otimes m).
\end{array} \right.
\]
Moreover, if $A$ is unital and its unit $1_A$ satisfies that for
any $a\in A$,
\[\left\{
\begin{array}{lll}
\sigma(a\otimes 1_A)&=&1_A\otimes a,\\[3pt] \sigma(1_A\otimes a)&=&a\otimes
1_A,
\end{array} \right.
\]then $A$ is called a \emph{unital braided algebra}.
\end{definition}

\begin{remark}1. For any braided vector space $(V,\sigma)$, there is a trivial braided algebra structure on it whose multiplication is the trivial one $m=0$.

2. The braided algebra structure is very crucial. Given a braided
vector space $(V,\sigma)$, it is not reasonable that there should
be a non-trivial braided algebra structure on $V$. For instance,
let $V$ be a vector space with basis $\{e_1,e_2\}$. We define two
braidings $\sigma_1$ and $\sigma_2$ on $V$ respectively by

\[\left\{
\begin{array}{lll}
\sigma_1(e_{1}\otimes e_{1})&=&e_{1}\otimes e_{1}, \\
\sigma_1(e_{i}\otimes e_{2})&=&qe_{2}\otimes e_{1},\\
\sigma_1(e_{2}\otimes e_{1})&=&qe_{1}\otimes
e_{2}+(1-q^{2})e_{2}\otimes e_{1},\\
\sigma_1(e_{2}\otimes e_{2})&=&e_{2}\otimes e_{2},
\end{array} \right.
\]
and
$$\sigma_2(e_{i}\otimes e_{j})=qe_{j}\otimes e_{i},\ \ \ \forall i,j,$$
where $q\in\mathbb{K}$ is nonzero and not equal to $\pm 1$.

Then by an easy argument, one can show that the only product on
$V$ which is compatible with $\sigma_1$ is just the trivial one.
The case is the same for $\sigma_2$.

3. For any braided algebra $(A,m,\sigma)$, one can embed it into a
unital braided algebra
$(\widetilde{A},\widetilde{m},\widetilde{\sigma})$ in the
following way. First of all, we set
$\widetilde{A}=\mathbb{K}\oplus A$. Then we define the
multiplication $\widetilde{m}$ and the braiding
$\widetilde{\sigma}$ by: for any $\lambda,\mu\in \mathbb{K}$ and
$a,b\in A$
$$\widetilde{m}\big((\lambda+a)\otimes (\mu+b)\big)=\lambda\mu+\lambda\cdot b+\mu\cdot
a+m(a\otimes b), $$ and
$$\widetilde{\sigma}\big((\lambda+a)\otimes (\mu+b)\big)=\mu\otimes
\lambda+b\otimes \lambda+\mu\otimes a+\sigma(a\otimes b).
$$It is easy to verify that $(\widetilde{A},\widetilde{m},\widetilde{\sigma})$ is a braided algebra with unit $1\in \mathbb{K}$.
\end{remark}

Because all the constructions in this paper are based on braided
algebras, we provide several concrete examples which will either
be used in our later discussion or afford the reader some
illustrations. Some of them may be known, while some may be new.
For more examples, one can see \cite{B} and \cite{JR}.

\begin{example}Let $V$ be a vector space with basis $\{e_i\}$ which is at most countable. We provide a braided algebra structure on $V$. The braiding $\sigma$ on $V$ is given by $\sigma(e_i\otimes
e_j)=q_{ij}e_j\otimes e_i$, where $q_{ij}$'s are nonzero scalars
in $\mathbb{K}$ such that $q_{ij}q_{ik}=q_{i\ j+k}$ and
$q_{ik}q_{jk}=q_{i+j\ k}$ for any $i,j,k$. For instance, let $q$
be a nonzero scalar in $\mathbb{K}$ and $q_{ij}=q^{ij}$. The
multiplication $\cdot$ on $V$ which is compatible with the
braiding $\sigma$ is given as follows.

Case 1. If $V$ is a finite-dimensional vector space with basis
$\{e_1,e_2,\ldots,e_N\}$, then we define

\[e_i\cdot e_j= \left\{
\begin{array}{lll}
e_{i+j},&& \mathrm{if\ }i+j\leq N,\\[3pt]
0,&&\mathrm{otherwise}.
\end{array} \right.
\]

Case 2. If $V$ is a vector space with basis $\{e_i\}_{i\in
\mathbb{N}}$, then we define $e_i\cdot e_j=e_{i+j}$ for any
$i,j\in \mathbb{N}$.

It is evident that $\cdot$ is an associative algebra structure on
$V$ in both cases. Notice that
\begin{eqnarray*}(\mathrm{id}_V\otimes
\cdot)\sigma_1\sigma_2 (e_i\otimes e_j\otimes
e_k)&=&q_{jk}q_{ik}e_k\otimes e_{i+j}\\[3pt]
&=&q_{i+j\ k}e_k\otimes e_{i+j}\\[3pt]&=&\sigma(
\cdot\otimes\mathrm{id}_V)(e_i\otimes e_j\otimes
e_k),\end{eqnarray*}
 and similarly
$(
\cdot\otimes\mathrm{id}_V)\sigma_2\sigma_1=\sigma(\mathrm{id}_V\otimes
\cdot)$. Therefore $(V,\cdot,\sigma)$ is a braided algebra.
\end{example}

\begin{example}All notions of this example can be found in \cite{Ka}. Let $q\neq 1$ be a invertible scalar in $\mathbb{K}$, and $x,y$ be two indeterminates. Denote by $\mathbb{K}_q[x,y]$ the quantum plane, i.e., the algebra generated by $x,y$ with the relation $yx=qxy$. It has a linear basis $\{x^iy^j\}_{i,j\geq 0}$. Define two algebra automorphisms $\omega_x$ and $\omega_y$ of $\mathbb{K}_q[x,y]$ by requiring that
$$\omega_x(x)=qx,\omega_x(y),\omega_y(x)=x,\omega_y(y)=qy,$$ and define two endomorphisms $\partial_q/\partial x$ and $\partial_q/\partial y$ by requiring that $$\frac{\partial_q(x^my^n)}{\partial x}=[m]x^{m-1}y^n,\frac{\partial_q(x^my^n)}{\partial y}=x^m[n]y^{n-1},$$ where $[k]=\frac{q^k-q^{-k}}{q-q^{-1}}$ for any $k\in \mathbb{N}$.

Let $U_q \mathfrak{sl}_2$ be the quantized algebra associated to
$\mathfrak{sl}_2$, i.e., the algebra generated by $E,F,K,K^{-1}$
with the relations $$KK^{-1}=K^{-1}K=1,$$
$$KE=q^2EK, \ KF=q^{-2}FK,$$  $$EF-FE=\frac{K-K^{-1}}{q-q^{-1}}.$$ It is well-known that $U_q \mathfrak{sl}_2$ is a quasi-triangular Hopf algebra.

By Theorem VII 3.3 in \cite{Ka}, $\mathbb{K}_q[x,y]$ is a $U_q
\mathfrak{sl}_2$-module-algebra with the following module
structure: for any $P\in \mathbb{K}_q[x,y] $,
$$EP=x\frac{\partial_q(P)}{\partial y}, FP=\frac{\partial_q(P)}{\partial x}y, $$ $$KP=(\omega_x\omega_y^{-1})(P), K^{-1}P=(\omega_y\omega_x^{-1})(P).$$
Set $V=\mathrm{Span}_\mathbb{K}\{x,y\}$. It is not hard to see
that the above action restricting on $V$ is the standard
2-dimensional simple $U_q \mathfrak{sl}_2$-module structure. We
know that (Theorem 2.7 in \cite{JR}) every module-algebra over a
quasi-triangular Hopf algebra has a braided algebra structure. So
$\mathbb{K}_q[x,y]$ is a braided algebra.\end{example}

\begin{example}Let $(V,\sigma)$ be a braided vector space.  It
is know that $(T(V),m,\beta)$ is a braided algebra, where $m$ is
the concatenation product. Let $M_n: V^{\otimes n}\rightarrow
T(V)$ be a linear map such that $\beta(M_n\otimes
\mathrm{id}_V)=(\mathrm{id}_V\otimes M_n)\beta_{n1}$. If we denote
by $\mathcal{I}$ the ideal of $T(V)$ generated by
$\mathrm{Im}M_n$, the image of $M_n$, then
$\beta\big(T(V)\underline{\otimes}\mathcal{I}+\mathcal{I}\underline{\otimes}T(V)\big)\subset
T(V)\underline{\otimes}I+I\underline{\otimes}T(V)$. So the
quotient algebra $T(V)/\mathcal{I}$ is also a braided algebra. For
instance, if $M_2=\mathrm{id}_V^{\otimes 2}-\sigma$, then the
quotient algebra is the $r$-symmetric algebra defined in \cite{B}.
\end{example}

\begin{example}Let $H$ be a finite dimensional quasi-triangular Hopf algebra. By a result of Majid (Theorem 3.3 in \cite{M}), the quantum double $\mathcal{D}(H)$ of $H$ is a braided algebra (according to a discussion in \cite{JR} for Radford's work \cite{Ra}).\end{example}

\section{Quantum quasi-shuffle algebras}
For any algebra $A$, it is a braided algebra with respect to the
flip map switching the two factors of $A\otimes A$. One can
construct an algebra structure on $T(A)$ which combines the the
multiplication of $A$ and the shuffle product of $T(A)$ (see
\cite{NR}). This structure is the the so-called quasi-shuffle
algebra. In fact, if the flip map is replaced by a braiding, one
can construct a quantized quasi-shuffle product by assuming some
compatibilities between the multiplication of $A$ and the braiding
(for more details, one can see \cite{JR} and \cite{JRZ}). Given a
braided algebra $(A, m,\sigma)$, the \emph{quantum quasi-shuffle
product} $\Join_\sigma$ on $T(A)$ is given by the following
inductive formula: for $i,j>1$ and any $a_1,\ldots,
a_i,b_1,\ldots, b_j\in A$,
\begin{eqnarray*}
\lefteqn{(a_1\otimes\cdots\otimes a_i)\Join_\sigma (b_1\otimes\cdots\otimes b_j)}\\
&=&a_1\otimes \big((a_2\otimes\cdots\otimes a_i)\Join_\sigma (b_1\otimes\cdots\otimes b_{j})\big)\\
&&+(\mathrm{id}_A\otimes \Join_{\sigma (i,j-1)})(\beta_{i,1}\otimes \mathrm{id}_A^{\otimes j-1})(a_1\otimes\cdots\otimes a_i\otimes b_1\otimes\cdots\otimes b_j)\\
&&+(m\otimes\Join_{\sigma (i-1,j-1)} )(\mathrm{id}_A\otimes
\beta_{i-1,1}\otimes \mathrm{id}_A^{\otimes
j-1})(a_1\otimes\cdots\otimes a_i\otimes b_1\otimes\cdots\otimes
b_j),
\end{eqnarray*}
where  $\Join_{\sigma (i,j)}$ denotes the restriction of
$\Join_\sigma$ on $V^{\otimes i}\underline{\otimes}V^{\otimes j}$.

\begin{remark}1. Given a braided algebra $(A, m,\sigma)$, $T_{\sigma}^{qsh}(A)=(T(A),\Join_\sigma)$ is a an associative algebra with unit $1\in \mathbb{K}$, and called the \emph{quantum quasi-shuffle algebra} built on $(A,
m,\sigma)$. Furthermore, the algebra $T_{\sigma}^{qsh}(A)$,
together with the braiding $\beta$ and the deconcatenation
coproduct $\delta$, forms a braided bialgebra in the sense of
\cite{Ta} (see \cite{JR}). The set of primitive elements is
exactly $A$.

2. By using Example 2.3, Proposition 17 in \cite{JRZ} can be
applied to any vector space whose basis is at most countable. In
other words, for any vector space $V$ with at most countable
basis, one can provide a linear basis of $T(V)$ by combining the
quantum quasi-shuffle product with Lyndon words.\end{remark}

\begin{example}[Hoffman's q-deformation]In \cite{Hof2}, Hoffman defined his q-deformation of
quasi-shuffles. It is an attempt to deform the quasi-shuffle
product according to the quantum shuffle product. Now we give an
explanation of the q-deformation from a point of view of the
quantum quasi-shuffles. Let $X$ be a locally finite set, i.e., $X$
is a disjoint union of finite set $X_n$, whose elements are called
letters of degree n, for $n\geq 1$. We denote by $\mathfrak{X}$
the vector space spanned by $X$. The elements in
$T(\mathfrak{X})$, which are of the form $a_1\otimes a_2\otimes
\cdots \otimes a_m$ with $a_i\in X$, are called words. Let $[ ,]$
be a graded associative product on $\mathfrak{X}$. Hoffman defined
an associative product $\ast_q$ on $T(\mathfrak{X})$: for any
words $w_1,w_2\in T(\mathfrak{X})$ and letters $a,b\in X $,
\begin{eqnarray*}\lefteqn{(a\otimes
w_1)\ast_q (b\otimes w_2)}\\
&=&a\otimes \big(w_1\ast_q(b\otimes w_2)\big)+q^{|a\otimes
w_1||b|}b\otimes \big((a\otimes w_1)\ast_q w_2
\big)+q^{|w_1||b|}[a,b](w_1\ast_q w_2),\end{eqnarray*}where $|w|$
is the degree of a word, i.e., the sum of degrees of its factors.

We define a braiding $\sigma$ on $\mathfrak{X}$ as follows: for
$x\in X_i$ and $y\in X_j$, $\sigma(x\otimes y)=q^{ij}y\otimes x$,
where $q\in K$ is a nonzero scalar. Since the product $[ ,]$
preserves the grading, it is an easy exercise to verify that
$(\mathfrak{X},[,],\sigma)$ is a braided algebra. By comparing
their reductive formulas, one can see that the quantum
quasi-shuffle algebra built on $(\mathfrak{X},[,],\sigma)$ is just
Hoffman's q-deformation.\end{example}

In order to give a more explicit description of the quantum
quasi-shuffle product, we need to recall some terminologies
introduced in \cite{GK}. An $(i,j)$-\emph{shuffle} is an element
$w\in \mathfrak{S}_{i+j}$ such that $w (1) < \cdots <w (i)$ and $w
(i+1) < \cdots <w (i+j)$. We denote by $\mathfrak{S}_{i,j}$ the
set of all $(i,j)$-shuffles. Given an $(i,j)$-shuffle $w$, a pair
$(k,k+1)$, where $1\leq k < i+j$, is called an \emph{admissible
pair} for $w$ if $w^{-1}(k)\leq i<w^{-1}(k+1)$. We denote by
$\mathcal{T}^w$ the set of all admissible pairs for $w$. For any
subset $S$ of $\mathcal{T}^w$, the pair $(w,S)$ is called a
\emph{mixable $(i,j)$-shuffle}. We denote by
$\overline{\mathfrak{S}}_{i,j}$ the set of all mixable
$(i,j)$-shuffles, i.e.,
$$\overline{\mathfrak{S}}_{i,j}=\{(w,S)|w\in \mathfrak{S}_{i,j}, S\subset \mathcal{T}^w\}.$$

Let $(A, m,\sigma)$ be a braided algebra. Define $m^k:A^{\otimes
k+1}\rightarrow A$ recursively by $m^0=\mathrm{id}_A$, $m^1=m$ and
$m^k=m(\mathrm{id}_A\otimes m^{k-1})$ for $k\geq 2$. Given
$n\in\mathbb{N}$, we denote
$$\mathcal{C}(n)=\{I=(i_1,\ldots,i_k)\in
\mathbb{N}^k|i_1+\cdots+i_k=n\}.$$
 The elements in $\mathcal{C}(n)$ are called \emph{compositions} of
 $n$. For any $I=(i_1,\ldots,i_k)\in \mathcal{C}(n)$, we define $m_I=m^{i_1-1}\otimes\cdots\otimes
 m^{i_k-1}$. For any $(w,S)\in\overline{\mathfrak{S}}_{i,j}$, we
 associate to $S$ a composition $\textrm{cp}(S)$ of $i+j$ as follows: if
$S=\{(k_1,k_1+1),\ldots, (k_s,k_s+1)\}$ with $k_1<\cdots <k_s$,
set  $$\textrm{cp}(S)=(\underbrace{1,\ldots\ldots,1}_{k_1-1\
\textrm{copies}},
2,\underbrace{1,\ldots\ldots\ldots\ldots,1}_{k_2-k_1-2\
\textrm{copies}},2,\ldots,2,\underbrace{1,\ldots\ldots\ldots,1}_{i+j-k_s-1\
\textrm{copies}}).$$  By convention, we set
$\textrm{cp}(\emptyset)=(1,1,\ldots,1)$. Denote
$T_{(w,S)}^\sigma=m_{\textrm{cp}(S)}\circ T_w^\sigma$.

\begin{theorem}Let $(A, m,\sigma)$ be a braided algebra. Then for any $a_1,\ldots,a_{i+j}\in A$,  $$(a_1\otimes\cdots\otimes a_i)\Join_\sigma (a_{i+1}\otimes\cdots\otimes a_{i+j})=\sum_{(w,S)\in\overline{\mathfrak{S}}_{i,j}}T_{(w,S)}^\sigma(a_1\otimes \cdots\otimes a_{i+j}).$$\end{theorem}
\begin{proof}We use induction on $i+j$.

When $i=j=1$,
$$a_1\Join_\sigma a_2=m(a_1\otimes a_2)+a_1\otimes a_2+\sigma(a_1\otimes a_2).$$
On the other hand,
$\overline{\mathfrak{S}}_{1,1}=\{(1_{\mathfrak{S}_2},\emptyset),
(1_{\mathfrak{S}_2},\{(1,2)\}), (s_1,\emptyset)\}$, where $s_1$ is
the generator of $\mathfrak{S}_2$. So the formula holds.

We assume that the formula is true in the case $\leq i+j$. By the
inductive formula of quantum quasi-shuffles, we have that
\begin{eqnarray*}\lefteqn{(a_1\otimes\cdots\otimes a_{i+1})\Join_\sigma (a_{i+2}\otimes\cdots\otimes a_{i+j+1})}\\[3pt]
&=&a_1\otimes \big((a_2\otimes\cdots\otimes a_{i+1})\Join_\sigma (a_{i+2}\otimes\cdots\otimes a_{i+j+1})\big)\\
&&+(\mathrm{id}_A\otimes \Join_{\sigma (i+1,j-1)})(\beta_{i+1,1}\otimes \mathrm{id}_A^{\otimes j-1})(a_1\otimes \cdots\otimes a_{i+j+1})\nonumber\\
&&+(m\otimes\Join_{\sigma (i,j-1)} )(\mathrm{id}_A\otimes
\beta_{i,1}\otimes \mathrm{id}_A^{\otimes
j-1})(a_1\otimes \cdots\otimes a_{i+j+1})\\[3pt]
&=&\sum_{(w,S)\in\overline{\mathfrak{S}}_{i,j}}(\mathrm{id}_A\otimes T_{(w,S)}^\sigma)(a_1\otimes \cdots\otimes a_{i+j+1})\\
&&+\sum_{(w,S)\in\overline{\mathfrak{S}}_{i+1,j-1}}(\mathrm{id}_A\otimes T_{(w,S)}^\sigma)(\beta_{i+1,1}\otimes \mathrm{id}_A^{\otimes j-1})(a_1\otimes \cdots\otimes a_{i+j+1})\nonumber\\
&&+\sum_{(w,S)\in\overline{\mathfrak{S}}_{i,j-1}}(m\otimes
T_{(w,S)}^\sigma )(\mathrm{id}_A\otimes \beta_{i,1}\otimes
\mathrm{id}_A^{\otimes j-1})(a_1\otimes \cdots\otimes
a_{i+j+1})\\[3pt]
&=&\sum_{(w,S)\in\overline{\mathfrak{S}}_{i,j}}(\mathrm{id}_A\otimes m_{\textrm{cp}(S)})T_{1_{\mathfrak{S}_1}\times w}^\sigma(a_1\otimes \cdots\otimes a_{i+j+1})\\
&&+\sum_{(w,S)\in\overline{\mathfrak{S}}_{i+1,j-1}}(\mathrm{id}_A\otimes m_{\textrm{cp}(S)})T_{(1_{\mathfrak{S}_1}\times w)\circ (\chi_{i+1,1}\times 1_{\mathfrak{S}_{j-1}})}^\sigma(a_1\otimes \cdots\otimes a_{i+j+1})\nonumber\\
&&+\sum_{(w,S)\in\overline{\mathfrak{S}}_{i,j-1}}(m\otimes
m_{\textrm{cp}(S)} )T^\sigma_{(1_{\mathfrak{S}_2}\times w)\circ
(1_{\mathfrak{S}_1}\times \chi_{i,1}\times
1_{\mathfrak{S}_{j-1}})}(a_1\otimes \cdots\otimes
a_{i+j+1}),\end{eqnarray*} where the third equality follows from
the fact that all the expressions of the permutations being lifted
are reduced.

Denote $$\mathcal{S}_1=\{(w,S)\in
\overline{\mathfrak{S}}_{i+1,j}|(1,2)\notin S, w(1)=1\},$$
$$\mathcal{S}_2=\{(w,S)\in \overline{\mathfrak{S}}_{i+1,j}|(1,2)\notin S, w(i+2)=1\},$$and
$$\mathcal{S}_3=\{(w,S)\in \overline{\mathfrak{S}}_{i+1,j}|(1,2)\in S\}.$$For any $(i+1,j)$-shuffle $w$, one has either $w(1)=1$ or $w(i+2)=1$.
Therefore $\mathcal{S}_1$, $\mathcal{S}_2$ and $ \mathcal{S}_3$
are mutually disjoint,  and
$\mathfrak{S}_{i+1,j}=\mathcal{S}_1\cup\mathcal{S}_2\cup\mathcal{S}_3$.

We make a further observation. It is easy to see that there is a
one-to-one correspondence between $\mathfrak{S}_{i,j}$ and $\{w\in
\mathfrak{S}_{i+1,j}|w(1)=1\}$ given by $w\mapsto
1_{\mathfrak{S}_1}\times w$ for any $w\in \mathfrak{S}_{i,j}$. So
$$\mathcal{S}_1=\{(1_{\mathfrak{S}_1}\times w,S)|w\in \mathfrak{S}_{i,j}, S\subset\mathcal{T}^{1_{\mathfrak{S}_1}\times w},(1,2)\notin S\}.$$
There is a one-to-one correspondence between
$\mathfrak{S}_{i+1,j-1}$ and $\{w\in
\mathfrak{S}_{i+1,j}|w(i+2)=1\}$ given by  $w\mapsto
\widetilde{w}=(1_{\mathfrak{S}_1}\times w)\circ
(\chi_{i+1,1}\times 1_{\mathfrak{S}_{j-1}})$ for any $w\in
\mathfrak{S}_{i+1,j-1}$. Consequently,
$$\mathcal{S}_2=\{(\widetilde{w},S)|w\in \mathfrak{S}_{i+1,j-1},S\subset\mathcal{T}^{\widetilde{w}},(1,2)\notin S\}.$$
Finally, for any $(w,S)\in \mathcal{S}_3$, we must have that
$w(1)=1$ and $w(i+2)=2$. There is a one-to-one correspondence
between $\mathfrak{S}_{i,j-1}$ and $\{w\in
\mathfrak{S}_{i+1,j}|w(1)=1,w(i+2)=2\}$ given by $w\mapsto
\overline{w}=(1_{\mathfrak{S}_2}\times w)\circ
(1_{\mathfrak{S}_1}\times \chi_{i,1}\times
1_{\mathfrak{S}_{j-1}})$ for any $w\in \mathfrak{S}_{i,j-1}$. So
$$\mathcal{S}_3=\{(\overline{w},S)\in \overline{\mathfrak{S}}_{i+1,j}|w\in \mathfrak{S}_{i,j-1},(1,2)\in S\}.$$

 As a conclusion, the three terms in the final step of the preceding
 computation come from $\mathcal{S}_1$, $\mathcal{S}_2$ and
 $\mathcal{S}_3$ respectively. So we have that $$(a_1\otimes\cdots\otimes a_{i+1})\Join_\sigma
(a_{i+2}\otimes\cdots\otimes
v_{i+j+1})=\sum_{(w,S)\in\overline{\mathfrak{S}}_{i+1,j}}T_{(w,S)}^\sigma(a_1\otimes
\cdots\otimes a_{i+j+1}),$$ which completes the induction.
 \end{proof}

\begin{remark}Let $(A,m)$ be an algebra and $\lambda$ be a scalar in $\mathbb{K}$. Then $(A, \lambda m)$ becomes a braided algebra with respect to the usual flip map. In this case, the formula in Theorem 3.3 coincides with the one of mixable shuffle product introduced in \cite{GK}.\end{remark}

Assume again that $(A,m,\sigma)$ is a braided algebra. We denote
by $S^{qsh}_\sigma(A)$ the subalgebra of $T^{qsh}_\sigma(A)$
generated by $A$. To describe this subalgebra, we need to
introduce some notations.

For a fixed $n\in \mathbb{N}$ and any $w\in \mathfrak{S}_n$, we
denote $$\mathcal{S}^w=\{(k,k+1)|1\leq k<n,
w^{-1}(k)<w^{-1}(k+1)\},$$and
$$\overline{\mathfrak{S}}_n=\{(w,S)|w\in \mathfrak{S}_n, S\subset
\mathcal{S}^w\}.$$For any $(w,S)\in \overline{\mathfrak{S}}_n$, we
associate to $S$ a composition $\textrm{cp}(S) $ of $n$ as
follows. Let $S=\{(k_1,k_1+1), \ldots, (k_s,k_s+1)\}$ with
$k_1<\cdots<k_s$. We divide $\{k_1,\ldots,k_s\}$ into several
subsets
$$\{k_1,\ldots,k_{i_1}\}, \{k_{i_1+1},\ldots,k_{i_1+i_2}\},
\ldots,\{k_{i_1+\cdots+i_{r-1}+1},\ldots,k_s\},$$ which obey the
rule that:
\[\left\{
\begin{array}{lll}
k_1+1=k_2,k_2+1=k_3,\ldots,k_{i_1-1}+1=k_{i_1},&&\\
k_{i_1+1}+1=k_{i_1+2},k_{i_1+2}+1=k_{i_1+3},\ldots,k_{i_1+i_2-1}+1=k_{i_2},&&\\
\ldots,&&\\
k_{i_1+\cdots+i_{r-1}+1}+1=k_{i_1+\cdots+i_{r-1}+2},\ldots,
k_{s-1}+1=k_{s},&&
\end{array} \right.
\]
 but $k_{i_1}+1<k_{i_1+1},
k_{i_1+i_2}+1<k_{i_1+i_2+1},\ldots,
k_{i_1+\cdots+i_{r-1}}+1<k_{i_1+\cdots+i_{r-1}+1}$. Denote
$i_r=s-i_1-\cdots-i_{r-1}$. Then we write
 $$\textrm{cp}(S)=(\underbrace{1,\ldots\ldots,1}_{k_1-1\ \textrm{copies}}, i_1+1,\underbrace{1,\ldots\ldots\ldots\ldots,1}_{k_{i_1+1}-k_{i_1}-1\ \textrm{copies}},i_2+1,\ldots,i_r+1,\underbrace{1,\ldots\ldots\ldots,1}_{n-k_{i_r}-1\ \textrm{copies}}).$$
We define as before the map
$T_{(w,S)}^\sigma=m_{\textrm{cp}(S)}\circ T_w^\sigma$ for any
$(w,S)\in\overline{\mathfrak{S}}_n$.

Now we provide a decomposition of $\overline{\mathfrak{S}}_{n+1}$
which will be used later. For any $1\leq i\leq n+1$, we denote
$\mathfrak{S}_{n+1}(i)=\{w\in \mathfrak{S}_{n+1}|w(1)=i\}$. It is
clear that $\mathfrak{S}_{n+1}$ is the disjoint union of all
$\mathfrak{S}_{n+1}(i)$'s, and for each $i$ there is a one-to-one
correspondence between $\mathfrak{S}_{n}$ and
$\mathfrak{S}_{n+1}(i)$ given by $w\mapsto
L(w,i)=(\chi_{1,i-1}\times
1_{\mathfrak{S}_{n+1-i}})\circ(1_{\mathfrak{S}_1}\times w) $ for
any $w\in \mathfrak{S}_{n-1}$. So
\begin{eqnarray*}\mathfrak{S}_{n+1}&=&\bigcup_{i=1}^{n+1}\mathfrak{S}_{n+1}(i)\\[3pt]
&=&\bigcup_{i=1}^{n+1}\bigcup_{w\in
\mathfrak{S}_{n}}\{L(w,i)\}\\[3pt]
&=&\bigcup_{w\in
\mathfrak{S}_{n}}\bigcup_{i=1}^{n+1}\{L(w,i)\}.\end{eqnarray*}
Then we have that
\begin{eqnarray*}\overline{\mathfrak{S}}_{n+1}&=&\bigcup_{w\in
\mathfrak{S}_{n}}\bigcup_{i=1}^{n+1}\{(L(w,i),S)|S\subset \mathcal{S}^{L(w,i)}\}\\[3pt]
&=&\big(\bigcup_{w\in
\mathfrak{S}_{n}}\bigcup_{i=1}^{n+1}\{(L(w,i),S)|S\subset
\mathcal{S}^{L(w,i)},(i,i+1)\notin
S\}\big)\\[3pt]
&&\cup\big(\bigcup_{w\in
\mathfrak{S}_{n}}\bigcup_{i=1}^{n+1}\{(L(w,i),S)|S\subset
\mathcal{S}^{L(w,i)},(i,i+1)\in S\}\big).\end{eqnarray*} All the
unions above are disjoint.

Given $w\in \mathfrak{S}_n$ and $S\subset \mathcal{S}^w$ with
$\mathrm{cp}(S)=(i_1,\ldots,i_s)$. For any $0\leq k\leq s$, we
denote
$$\mathrm{cp}(S)_k=(i_1,\ldots,i_k,1,i_{k+1},\ldots,i_s),$$
$$\mathrm{cp}(S)^k=(i_1,\ldots,i_{k-1},i_k+1,i_{k+1},\ldots,i_s),$$
and $$I_k=(\underbrace{1,\ldots,1}_{k\
\textrm{copies}},2,\underbrace{1,\ldots\ldots\ldots,1}_{n-1-|S|-k\
\textrm{copies}}).$$Here, $|S|$ denotes the cardinality of the set
$S$.

\begin{lemma}Under the assumptions above, we have \[\left\{
\begin{array}{lll}
(\beta_{1,k}\otimes \mathrm{id}_A^{\otimes
n-|S|-k})(\mathrm{id}_A\otimes
T_{(w,S)}^\sigma)&=&m_{\mathrm{cp}(S)_k}T_{L(w,i_1+\cdots+i_k+1)},\\[5pt]
m_{I_k}(\beta_{1,k}\otimes \mathrm{id}_A^{\otimes
n-|S|-k})(\mathrm{id}_A\otimes
T_{(w,S)}^\sigma)&=&m_{\mathrm{cp}(S)^k}T_{L(w,i_1+\cdots+i_k+1)}.
\end{array} \right.
\]\end{lemma}
\begin{proof}Since $\sigma(\mathrm{id}_A\otimes m^l)=(m^l\otimes
\mathrm{id}_A)\beta_{1,l+1}$ for any $l$ (see Lemma 2 in
\cite{B}), \begin{eqnarray*}\lefteqn{(\beta_{1,k}\otimes
\mathrm{id}_A^{\otimes n-|S|-k})(\mathrm{id}_A\otimes
T_{(w,S)}^\sigma)}\\[3pt]
&=&(\beta_{1,k}\otimes \mathrm{id}_A^{\otimes
n-|S|-k})(\mathrm{id}_A\otimes
m_{\textrm{cp}(S)})(\mathrm{id}_A\otimes T_w^\sigma)\\[3pt]
&=&m_{\textrm{cp}(S)_k}(\beta_{1,i_1+\cdots+i_k}\otimes
\mathrm{id}_A^{\otimes i_{k+1}+\cdots+i_s})(\mathrm{id}_A\otimes T_w^\sigma)\\[3pt]
&=&m_{\textrm{cp}(S)_k}T_{(\chi_{1,i_1+\cdots+i_k}\times
1_{\mathfrak{S}_{i_{k+1}+\cdots+i_s}})\circ(1_{\mathfrak{S}_1}\times
w)}^\sigma\\[3pt]
&=&m_{\mathrm{cp}(S)_k}T_{L(w,i_1+\cdots+i_k+1)}.\end{eqnarray*}

The second equality is a consequence of the first one.\end{proof}

\begin{theorem}Let $(A,m,\sigma)$ be a braided algebra. For any $a_1,\ldots,a_n\in A$, we have that$$a_1\Join_\sigma\cdots\Join_\sigma a_n=\sum_{(w,S)\in\overline{\mathfrak{S}}_n}T_{(w,S)}^\sigma(a_1\otimes \cdots\otimes a_{n}).$$ Therefore $S^{qsh}_\sigma(A)=\sum_{n\geq 0}\mathrm{Im}(\sum_{(w,S)\in\overline{\mathfrak{S}}_n}T_{(w,S)}^\sigma)$.\end{theorem}
\begin{proof}We use induction on $n$.

When $n=2$, it is trivial since
$$\overline{\mathfrak{S}}_2=\{(1_{\mathfrak{S}_2},\emptyset),(1_{\mathfrak{S}_2},\{(1,2)\}),(s_1,\emptyset)
\}.$$ By Theorem 3.5, we have that for any $a_1,\ldots,a_{r+1}\in
A$,
\begin{eqnarray*}\lefteqn{a_1\Join_\sigma (a_2\otimes\cdots\otimes a_{r+1})}\\[3pt]
&=&\sum_{k=0}^r(\beta_{1,k}\otimes \mathrm{id}_A^{\otimes r-k})(a_1\otimes \cdots\otimes a_{r+1})\\[3pt]
&&+\sum_{k=0}^{r-1}(\mathrm{id}_A^{\otimes k}\otimes m\otimes
\mathrm{id}_A^{\otimes r-k-1})(\beta_{1,k}\otimes
\mathrm{id}_A^{\otimes r-k})(a_1\otimes \cdots\otimes
a_{r+1}).\end{eqnarray*} Therefore,
\begin{eqnarray*}\lefteqn{a_1\Join_\sigma\cdots\Join_\sigma
a_{n+1}}\\[3pt]
&=&a_1\Join_\sigma\big(\sum_{(w,S)\in\overline{\mathfrak{S}}_n}T_{(w,S)}^\sigma(a_2\otimes
\cdots\otimes a_{n+1})\big)\\[3pt]
&=&\sum_{(w,S)\in\overline{\mathfrak{S}}_n}\sum_{k=0}^{n-|S|}(\beta_{1,k}\otimes
\mathrm{id}_A^{\otimes n-|S|-k})\big(a_1\otimes
T_{(w,S)}^\sigma(a_2\otimes
\cdots\otimes a_{n+1})\big)\\[3pt]
&&+\sum_{(w,S)\in\overline{\mathfrak{S}}_n}\sum_{k=0}^{n-|S|-1}(\mathrm{id}_A^{\otimes
k}\otimes m\otimes \mathrm{id}_A^{\otimes
n-|S|-k-1})\\[3pt]
&&\ \ \ \ \ \ \ \ \ \ \ \ \ \ \ \ \ \ \  \ \ \ \ \ \ \ \ \
\circ(\beta_{1,k}\otimes \mathrm{id}_A^{\otimes
n-|S|-k})\big(a_1\otimes T_{(w,S)}^\sigma(a_2\otimes \cdots\otimes
a_{n+1})\big)\\[3pt]
&=&\sum_{(w,S)\in\overline{\mathfrak{S}}_n}\sum_{k=0}^{n-|S|}m_{\textrm{cp}(S)_k}T_{L(w,i_1+\cdots+i_k+1)}^\sigma(a_1\otimes
\cdots\otimes a_{n+1})\\[3pt]
&&+
\sum_{(w,S)\in\overline{\mathfrak{S}}_n}\sum_{k=0}^{n-|S|-1}m_{\textrm{cp}(S)^k}T_{L(w,i_1+\cdots+i_k+1)}^\sigma(a_1\otimes
\cdots\otimes a_{n+1}),\end{eqnarray*} where the last equality
follows from the preceding lemma.

On the other hand, by the decomposition of
$\overline{\mathfrak{S}}_{n+1}$ mentioned before,
\begin{eqnarray*}\lefteqn{\sum_{(w,S)\in\overline{\mathfrak{S}}_{n+1}}T_{(w,S)}^\sigma(a_1\otimes
\cdots\otimes a_{n+1})}\\[3pt]
&=&\sum_{w\in\mathfrak{S}_n}\sum_{i=1}^{n+1}\sum_{\substack{S\subset
\mathcal{S}^{L(w,i)}\\(i,i+1)\notin S}}T_{(L(w,i),S)}^\sigma
(a_1\otimes
\cdots\otimes a_{n+1})\\[3pt]
&&+\sum_{w\in\mathfrak{S}_n}\sum_{i=1}^{n+1}\sum_{\substack{S\subset
\mathcal{S}^{L(w,i)}\\(i,i+1)\in
S}}T_{(L(w,i),S)}^\sigma(a_1\otimes \cdots\otimes
a_{n+1}).\end{eqnarray*}We compare the terms in these two
expressions. Notice that for a fixed $1\leq i\leq n+1$ and
$S\subset \mathcal{S}^{L(w,i)}$ with $(i,i+1)\notin S$, there is a
unique $S'\in \mathcal{S}^w$ such that
$\textrm{cp}(S)=\textrm{cp}(S')_i$. Indeed, we can write down
$S^\prime$ explicitly: if $S=\{(k_1,k_1+1), \ldots, (k_s,k_s+1)\}$
with $k_1<\cdots<k_l<i<k_{l+1}<\cdots<k_s$, then
$S'=\{(k_1,k_1+1), \ldots,
(k_l,k_l+1),(k_{l+1}-1,k_{l+1}),\ldots,(k_s-1,k_s)\}$. Similarly,
if $(i,i+1)\in S$, there is a unique $S''\in \mathcal{S}^w$ such
that $\textrm{cp}(S)=\textrm{cp}(S'')^i$. It follows that every
term in
$\sum_{(w,S)\in\overline{\mathfrak{S}}_{n+1}}T_{(w,S)}^\sigma(a_1\otimes
\cdots\otimes a_{n+1})$ is from exactly one term in the formula of
$a_1\Join_\sigma\cdots\Join_\sigma a_{n+1}$. The converse is also
true. Since all terms in each formula are mutually distinct, we
get the conclusion.
\end{proof}

\begin{remark}Consider Example 3.2 and let $(\mathfrak{X},[,],\sigma)$ be the braided algebra introduced there. For any $w\in \mathfrak{S}_n$, we denote $\iota(w)=\{(i,j)|1\leq i<j\leq n, w(i)>w(j)\}$. Then for any $a_1,\cdots, a_n\in X$, $$T_w^\sigma(a_1\otimes\cdots\otimes a_n)=q^{\sum_{(i,j)\in \iota(w)}|a_i||a_j|}a_{w^{-1}(1)}\otimes \cdots \otimes a_{w^{-1}(n)}.$$

For any two compositions $I=(i_1,\ldots,i_k)$ and
$J=(j_1,\ldots,j_l)$ of $n$, we say $I$ is a \emph{refinement} of
$J$, written by $I\succeq J$, if there are $r_1,\ldots,r_l\in
\mathbb{N}$ such that $r_1+\cdots +r_l=k$ and
 $$i_1+\cdots+i_{r_1}=j_1, i_{r_1+1}+\cdots+i_{r_1+r_2}=j_2,\ldots, i_{r_1+\cdots+r_{l-1}+1}+\cdots+i_k=j_l.$$
 For instance, $(1,2,2,3)\succeq (3,2,3)$. For any $w\in \mathfrak{S}_n$, let $C(w)$ be the composition $(i_1,\ldots,i_k)$ of $n$ such that $$\{i_1,i_1+i_2,\ldots, i_1+\cdots+i_{k-1}\}=\{l|1\leq l\leq n-1, w(l)>w(l+1)\}.$$
For any $I=(i_1,\ldots,i_l)\in \mathcal{C}(n)$, we write
$I[a_1\otimes \cdots a_n]=[,]_I(a_1\otimes \cdots a_n)$.

Observing that for any $w\in \mathfrak{S}_n$ there is a one-to-one
correspondence between $S\subset \mathcal{S}^w$ and $I\in
\mathcal{C}(n)$ with $I\succeq C(w)$, one has immediately that
$$a_1\ast_q\cdots\ast_q
a_n=\sum_{w\in\mathfrak{S}_n}q^{\sum_{(i,j)\in
\iota(w)}|a_i||a_j|}\sum_{I\succeq C(w)}I[a_{w^{-1}(1)}\otimes
\cdots \otimes a_{w^{-1}(n)}].$$ This formula is given by Hoffman
when $[,]$ is commutative (see Lemma 5.2 in \cite{Hof2}).
\end{remark}

We conclude this section by an interesting formula. Let $V$ be a
vector space with basis $\{e_i\}_{i\in \mathbb{N}}$, and
$(V,m,\sigma)$ be the braided algebra structure given in Example
2.3. So we have that $\sigma(e_i\otimes e_j)=q_{ij}e_j\otimes e_i$
and $m(e_i\otimes e_j)=e_{i+j}$.

For fixed $i$, we denote $(0)_{q_{ii}}=1$ and
$(n)_{q_{ii}}=1+q_{ii}+q_{ii}^2+\cdots+q_{ii}^{n-1}=\frac{1-q_{ii}^n}{1-q_{ii}}$
when $n\in \mathbb{N}$. We denote
$(n)_{q_{ii}}!=(1)_{q_{ii}}\cdots (n)_{q_{ii}}$.

\begin{proposition}For nay $i,k\in \mathbb{N}$, we have that $$e_i^{\Join_\sigma k}=\sum_{n=1}^k\sum_{\substack{l_1+\cdots+l_n=k\\1\leq l_1,\ldots,l_n\leq k}}\frac{(k)_{q_{ii}}!}{(l_1)_{q_{ii}}!\cdots (l_n)_{q_{ii}}!}e_{l_1i}\otimes \cdots \otimes e_{l_ni}.$$\end{proposition}
\begin{proof}It is a direct verification by using induction and Theorem 3.6 or Hoffman's formula.\end{proof}

\section{The dual construction}
In this section, we give a dual construction of the quantum
quasi-shuffle algebra by using the universal property of tensor
algebras. First of all, we study a special coalgebra structure on
$T(C)$ which is a generation of the quantized cofree coalgebra
structure. For coalgebras, we adopt Sweedler's notation. That
means for a coalgebra $(C,\bigtriangleup,\varepsilon)$ and any
$c\in C$, we denote
$$\bigtriangleup(c)=\sum_{(c)}c_{(1)}\otimes c_{(2)},$$
or simply, $\bigtriangleup(c)=c_{(1)}\otimes c_{(2)}$.

To extend algebra structures to quantized case, one needs the
notion of braided algebras. By contrast, in the case of
coalgebras, one needs the so-called braided coalgebras.

\begin{definition}Let $C=(C,\bigtriangleup,\varepsilon)$ be a coalgebra with coproduct $\bigtriangleup$ and counit $\varepsilon$, and $\sigma$ be a braiding on $C$. We call $(C,\bigtriangleup,\sigma)$ a \emph{braided coalgebra} if it satisfies the following conditions:
\[\left\{
\begin{array}{lll}
(\mathrm{id}_C\otimes \bigtriangleup)\sigma&=&\sigma_1\sigma_2(
\bigtriangleup\otimes
\mathrm{id}_C),\\[3pt] ( \bigtriangleup\otimes
\mathrm{id}_C)\sigma&=&\sigma_2\sigma_1(\mathrm{id}_C\otimes
\bigtriangleup).
\end{array} \right.
\]

\end{definition}

In order to get new braided algebras and coalgebras from old ones,
we need the proposition below.
\begin{proposition}[\cite{HH}, Proposition 4.2]1.
For a braided algebra $(A,\mu,\sigma)$ and any $i\in \mathbb{N}$,
$(A^{\otimes i},\mu_{\sigma,i}, \beta_{ii})$ becomes a braided
algebra with product $\mu_{\sigma,i}=\mu^{\otimes i}\circ
T^\sigma_{w_i}$, where $w_i\in \mathfrak{S}_{2i}$ is given by
\[w_{i}=\left(\begin{array}{ccccccccc}
1&2&3&\cdots&i&i+1&i+2&\cdots & 2i\\
1&3&5&\cdots&2i-1&2& 4 &\cdots & 2i
\end{array}\right).\]

2. For a braided coalgebra $(C, \bigtriangleup,\sigma)$,
$(C^{\otimes i},\bigtriangleup_{\sigma,i}, \beta_{ii})$ becomes a
braided coalgebra with coproduct $\bigtriangleup_{\sigma,i}=
T^\sigma_{w_i^{-1}}\circ\bigtriangleup^{\otimes i}$ and counit
$\varepsilon^{\otimes i}:C^{\otimes i}\rightarrow
\mathbb{K}^{\otimes i}\simeq \mathbb{K}$.\end{proposition}

Let $(C,\bigtriangleup,\sigma)$ be a braided coalgebra. Consider
the tensor algebra $T(C)$ with the concatenation product $m$. Then
by
 the above proposition, $\mathscr{T} _\beta^2(C)=(T(C)\underline{\otimes}
 T(C),m_{\beta,2})$ and $\mathscr{T} _\beta^3=(T(C)\underline{\otimes}
 T(C)\underline{\otimes}
 T(C),m_{\beta,3})$ are associative algebras.

 We define
\[\begin{array}{cccc}
\phi_1: &C&\rightarrow& T(C)\underline{\otimes}
 T(C)
,\\[6pt]
&c&\mapsto&1\underline{\otimes} c+ c\underline{\otimes}1.
\end{array}\]

By the universal property of tensor algebras, there exists an
algebra map $\Phi_1: T(C)\rightarrow \mathscr{T} _\beta^2(C)$
whose restriction on $C$ is $\phi_1$. Moreover $\Phi_1$ is
coassociative and is the dual of quantum shuffle product (see,
e.g., \cite{FG}).

Now we define
\[\begin{array}{cccc}
\phi_2: &C&\rightarrow& T(C)\underline{\otimes}
 T(C)
,\\[6pt]
&c&\mapsto&\sum c_{(1)}\underline{\otimes} c_{(2)}.
\end{array}\]

By the universal property of tensor algebras again, there exists
an algebra map $\Phi_2: T(C)\rightarrow \mathscr{T} _\beta^2(C)$
whose restriction on $C$ is $\phi_2$.

\begin{proposition}For any $i\in \mathbb{N}$, we have that $\Phi_2\mid_{C^{\otimes
i}}=\bigtriangleup_{\sigma,i}$. So $(T(C), \Phi_2, \beta)$ is a
braided coalgebra.\end{proposition}
\begin{proof}We use induction on $i$. When $i=1$, it is trivial. We assume the equality holds for the case $i< n$.
Then for any $c_1,\ldots,c_n\in C$, we have
\begin{eqnarray*}\lefteqn{\Phi_2(c_1\otimes \cdots \otimes
c_n)}\\[3pt]
&=&m_{\beta,2}\big(\Phi_2(c_1)\otimes\Phi_2(c_2\otimes \cdots \otimes c_n)\big)\\[3pt]
&=&(\mathrm{id}_{T(C)}\otimes \beta \otimes \mathrm{id}_{T(C)})(\bigtriangleup(c_1)\otimes T^\sigma_{w_{n-1}^{-1}}\circ\bigtriangleup^{\otimes n-1}(c_2\otimes \cdots \otimes c_n)\big)\\[3pt]
&=&(\mathrm{id}_C\otimes T^\sigma_{\chi_{1,n-1}}\otimes
\mathrm{id}_C^{\otimes n-1})(\mathrm{id}_C^{\otimes 2}\otimes
T^\sigma_{w_{n-1}^{-1}})\bigtriangleup^{\otimes n}(c_1\otimes
\cdots \otimes
c_n)\\[3pt]
&=&T^\sigma_{w_n^{-1}}\circ\bigtriangleup^{\otimes n}(c_2\otimes
\cdots \otimes c_n),\end{eqnarray*}where the last equality follows
from the fact that $(1_{\mathfrak{G}_1 }\times\chi_{1,n-1}\times
1_{\mathfrak{G}_{}n-1 })(1_{\mathfrak{G}_2 }\times
w_{n-1}^{-1})=w_{n}^{-1}$ and the expression is
reduced.\end{proof}

Let $\phi=\phi_1+\phi_2: C\rightarrow\mathscr{T} _\beta^2(C)$ and
$\Phi$ be the algebraic map induced by the universal property of
tensor algebras which extends $\phi$.

\begin{proposition}Under the notation above, the triple $(T(C), \Phi, \beta)$ is a braided coalgebra.
\end{proposition}
\begin{proof}We first show that
\[\left\{
\begin{array}{lll}
\beta_1\beta_2 (\Phi\otimes \mathrm{id}_{T(C)})&=& (\mathrm{id}_{T(C)}\otimes \Phi)\beta,\\[3pt]
\beta_2\beta_1(\mathrm{id}_{T(C)}\otimes \Phi)&=&(\Phi\otimes
\mathrm{id}_{T(C)})\beta.
\end{array} \right.
\]
For any $x\in C^{\otimes i}$ and $y\in C^{\otimes j}$ we verify
the first one on $x\underline{\otimes}y$. The second one can be
verified similarly. We use induction on $i$.

When $i=1$,
\begin{eqnarray*}\beta_1\beta_2 (\Phi\otimes \mathrm{id}_{T(C)})(x\underline{\otimes}y)&=&\beta_1\beta_2 (\phi_1\otimes \mathrm{id}+\phi_2\otimes \mathrm{id})(x\underline{\otimes}y)\\[3pt]
&=&(\mathrm{id}\otimes \phi_1+\mathrm{id}\otimes \phi_2)\beta(x\underline{\otimes}y)\\[3pt]
&=& (\mathrm{id}_{T(C)}\otimes
\Phi)\beta(x\underline{\otimes}y).\end{eqnarray*} For any $c\in
C$, we have
\begin{eqnarray*}\beta_1\beta_2 (\Phi\otimes \mathrm{id}_{T(C)})\big((c\otimes x)\underline{\otimes}y\big)&=&\beta_1\beta_2\beta_3\beta_4\beta_2(\Phi\otimes \Phi\otimes\mathrm{id}_{T(C)})(v\underline{\otimes} x\underline{\otimes}y)\\[3pt]
&=&\beta_1\beta_3\beta_2\beta_3\beta_4(\Phi\otimes \Phi\otimes\mathrm{id}_{T(C)})(c\underline{\otimes} x\underline{\otimes}y)\\[3pt]
&=&\beta_1\beta_3\beta_2\big(\Phi\otimes \beta_1\beta_2(\Phi\otimes\mathrm{id}_{T(C)})\big)(c\underline{\otimes} x\underline{\otimes}y)\\[3pt]
&=&\beta_1\beta_3\beta_2(\Phi\otimes \mathrm{id}_{T(C)}\otimes\Phi)\beta_2(v\underline{\otimes} x\underline{\otimes}y)\\[3pt]
&=&\beta_3\big(\beta_1\beta_2(\Phi\otimes \mathrm{id}_{T(C)})\otimes\Phi\big)\beta_2(c\underline{\otimes} x\underline{\otimes}y)\\[3pt]
&=&\beta_3(\mathrm{id}_{T(C)}\otimes \Phi\otimes\Phi)\beta_1\beta_2(c\underline{\otimes} x\underline{\otimes}y)\\[3pt]
&=&(\mathrm{id}_{T(C)}\otimes \Phi)\beta\big((c\otimes
x)\underline{\otimes}y\big).\end{eqnarray*}

The next step is to show $(\Phi\otimes
\mathrm{id}_{T(C)})\Phi=(\mathrm{id}_{T(C)}\otimes \Phi)\Phi$.
Notice that for any $c\in C$,
\begin{eqnarray*}\lefteqn{(\Phi\otimes \mathrm{id}_{T(C)})\Phi(c)}\\[3pt]
&=&(\Phi\otimes \mathrm{id}_{T(C)})(c_{(1)}\underline{\otimes} c_{(2)}+1\underline{\otimes} c+c\underline{\otimes} 1)\\[3pt]
&=&c_{(1)}\underline{\otimes} c_{(2)}\underline{\otimes}
c_{(3)}+1\underline{\otimes} c_{(1)}\underline{\otimes} c_{(2)}+
c_{(1)}\underline{\otimes} c_{(2)}\underline{\otimes}
1\\[3pt]
&&+c_{(1)}\underline{\otimes}1 \underline{\otimes}
c_{(2)}+1\underline{\otimes} 1
\underline{\otimes}c+c_{(1)}\underline{\otimes} c_{(2)}\underline{\otimes} 1+1\underline{\otimes}c \underline{\otimes}1+c\underline{\otimes}1\underline{\otimes}1\\[3pt]
&=&(\mathrm{id}_{T(C)}\otimes \Phi)\Phi(c).\end{eqnarray*} By the
uniqueness of the universal property of $T(C)$, we only need to
show that both $(\Phi\otimes \mathrm{id}_{T(C)})\Phi$ and
$(\mathrm{id}_{T(C)}\otimes \Phi)\Phi$ are algebra morphisms from
$T(C)$ to $\mathscr{T} _\beta^3$. Since $\Phi: T(C)\rightarrow
\mathscr{T} _\beta^2(C)$ is an algebra morphism, we have that
$$\Phi\circ m=(m\otimes m)(\mathrm{id}_{T(C)}\otimes \beta
\otimes \mathrm{id}_{T(C)})(\Phi\otimes \Phi).$$ So
\begin{eqnarray*}
\lefteqn{(\Phi\otimes \mathrm{id}_{T(C)})\circ\Phi\circ
m}\\[3pt]
&=&(\Phi\otimes \mathrm{id}_{T(C)})(m\otimes
m)(\mathrm{id}_{T(C)}\otimes \beta \otimes
\mathrm{id}_{T(C)})(\Phi\otimes \Phi)\\[3pt]
&=&\big((\Phi\circ m)\otimes m\big)(\mathrm{id}_{T(C)}\otimes
\beta \otimes
\mathrm{id}_{T(C)})(\Phi\otimes \Phi)\\[3pt]
&=&(m\otimes m \otimes m)(\mathrm{id}_{T(C)}\otimes \beta \otimes
\mathrm{id}_{T(C)} \otimes \mathrm{id}_{T(C)}\otimes
\mathrm{id}_{T(C)})\\[3pt]
&&\circ(\Phi\otimes \Phi  \otimes \mathrm{id}_{T(C)} \otimes
\mathrm{id}_{T(C)})(\mathrm{id}_{T(C)}\otimes \beta \otimes
\mathrm{id}_{T(C)})(\Phi\otimes \Phi)\\[3pt]
&=&(m\otimes m\otimes m) \beta_2\\[3pt]
&&\circ\big(\mathrm{id}_{T(C)} \otimes \mathrm{id}_{T(C)}\otimes(\Phi\otimes \mathrm{id}_{T(C)}) \beta \otimes \mathrm{id}_{T(C)}\big)\big((\Phi\otimes \mathrm{id}_{T(C)})\Phi\otimes \Phi\big)\\[3pt]
&=&(m\otimes m\otimes
m)\beta_2\beta_4\beta_3\\[3pt]
&&\circ(\mathrm{id}_{T(C)}\otimes \mathrm{id}_{T(C)}\otimes
\mathrm{id}_{T(V)}\otimes \Phi\otimes
\mathrm{id}_{T(C)})\big((\Phi\otimes \mathrm{id}_{T(C)})\Phi\otimes \Phi\big)\\[3pt]
&=&m_{\beta,3}\big((\Phi\otimes \mathrm{id}_{T(C)})\Phi\otimes(
\Phi_2\otimes \mathrm{id}_{T(C)})\Phi\big).
\end{eqnarray*}
It follows that $(\Phi\otimes \mathrm{id}_{T(C)})\Phi$ is an
algebra morphism. Similarly, the map $(\mathrm{id}_{T(C)}\otimes
\Phi)\Phi$ is also an algebra morphism.
\end{proof}

Now we begin to study the relation between the braided coalgebra
$(T(C), \Phi, \beta)$ and the quantum quasi-shuffle algebra. We
show that they are dual to each other in the following sense.

Let $<,>: V\times W\rightarrow \mathbb{K}$ and
$<,>^\prime:V^\prime\times W^\prime\rightarrow \mathbb{K}$ be two
bilinear non-degenerate forms. For any $f\in \mathrm{Hom}(V,
V^\prime)$, the adjoint operator $\mathrm{adj}(f)\in
\mathrm{Hom}(W^\prime, W)$ of $f$ is defined to be the one such
that $<x,\mathrm{adj}(f)(y)>=<f(x),y>^\prime$ for any $x\in V$ and
$y\in W^\prime$. It is clear that $\mathrm{adj}(f\circ
g)=\mathrm{adj}(g)\circ\mathrm{adj}(f)$.

\begin{remark}If there is a non-degenerate bilinear form between two vector spaces $A$ and $C$, and $(A,m,\sigma)$ is a braided algebra, then
$(C,\mathrm{adj}(m),\mathrm{adj}(\sigma))$ is a braided coalgebra.
The converse is also true.\end{remark}

From now on, we assume that there always exists a non-degenerate
bilinear form $<,>$ between two vector spaces $A$ and $C$. It can
be extended to a bilinear form $<,>: A^{\otimes n}\times
C^{\otimes n}\rightarrow \mathbb{K}$ for any $n \geq 1$ in the
usual way: for any $a_1,\ldots, a_n\in A$ and $c_1,\ldots,c_n\in
C$,
$$<a_1\otimes \cdots \otimes a_n, c_1\otimes \cdots \otimes
c_n>=\prod_{i=1}^n <a_i, c_i>.$$ It induces a non-degenerate
bilinear form  $<,>: T(A)\times T(C)\rightarrow \mathbb{K}$ by
setting that $<x,y>=0$ for any $x\in A^{\otimes i}$ , $y\in
C^{\otimes j}$ and $i\neq j$. Then we can define a non-degenerate
bilinear form $<,>: T(A)\underline{\otimes}T(A)\times
T(C)\underline{\otimes}T(C)\rightarrow \mathbb{K}$ by requiring
that $<u\underline{\otimes} v,x\underline{\otimes}y>=<u,x><v,y>$
for any $u,v\in T(A)$ and $x,y\in T(C)$.

If $(C,\bigtriangleup,\sigma)$ is a braided coalgebra, we denote
$\tau=\mathrm{adj}(\sigma)$ and $\alpha=\mathrm{adj}(\beta)$. Then
$\alpha$ is a braiding on $T(A)$ and
$\alpha_{i,j}=T^\tau_{\chi_{ji}^{-1}}=T^\tau_{\chi_{ij}}$.

\begin{theorem}Under the assumptions above, we have that $\mathrm{adj}(\Phi)=\Join_\sigma$.\end{theorem}
\begin{proof}For any $c_1,\ldots,c_n\in
C$, we notice that
\begin{eqnarray*}\Phi(c_1\otimes \cdots\otimes c_n)&=&m_{\beta,2}\big(\Phi(c_1)\underline{\otimes}\Phi(c_2\otimes \cdots\otimes c_n\big)\\[3pt]
&=&m_{\beta,2}\big((1\underline{\otimes}c_1)\underline{\otimes}\Phi(c_2\otimes \cdots\otimes c_n)\big)\\[3pt]
&&+m_{\beta,2}\big((v_1\underline{\otimes}1)\underline{\otimes}\Phi(c_2\otimes \cdots\otimes c_n)\big)\\[3pt]
&&+m_{\beta,2}\big((c_{1(1)}\underline{\otimes}c_{1
(2)})\underline{\otimes}\Phi(c_2\otimes \cdots\otimes
c_n)\big)\\[3pt]
&=&(\beta_{1,?}\otimes \mathrm{id}_C^{\otimes n-1-?})(\mathrm{id}_C\otimes \Phi)(c_1\otimes \cdots\otimes c_n)\\[3pt]
&&+(\mathrm{id}_C\otimes \Phi)(c_1\otimes \cdots\otimes c_n)\\[3pt]
&&+(\mathrm{id}_C\otimes\beta_{1,?}\otimes \mathrm{id}_C^{\otimes
n-1-?})(\bigtriangleup\otimes \Phi)(c_1\otimes \cdots\otimes
c_n).\end{eqnarray*} Here, since $\Phi(C^{\otimes k})\subset
C^{\otimes 0}\underline{\otimes}C^{\otimes k}+C^{\otimes
k}\underline{\otimes}C^{\otimes 0}+\cdots+C^{\otimes
k}\underline{\otimes}C^{\otimes k}$, we denote by $\beta_{1,?}$
the action of $\beta$ on $C\underline{\otimes}C^{\otimes ?}$ where
$C^{\otimes ?}$ is the left factor of some component in
$\Phi(C^{\otimes k})$.

We denote by $\mathrm{adj}(\Phi)_{(k,l)}$ the action of
$\mathrm{adj}(\Phi)$ on $A^{\otimes
k}\underline{\otimes}A^{\otimes l}$. Then on $A^{\otimes
i}\underline{\otimes}A^{\otimes j}$, we have
\begin{eqnarray*}\mathrm{adj}(\Phi)_{(i,j)}&=&\mathrm{adj}\big((\beta_{1,?}\otimes
\mathrm{id}_C^{\otimes i+j-1-?})(\mathrm{id}_C\otimes
\Phi)+(\mathrm{id}_C\otimes
\Phi)\\[3pt]
&&\ \ \ \ \ \ +(\mathrm{id}_C\otimes\beta_{1,?}\otimes
\mathrm{id}_C^{\otimes i+j-1-?})(\bigtriangleup\otimes \Phi)\big)_{(i,j)}\\[3pt]
&=&\big(\mathrm{id}_A\otimes \mathrm{adj}(\Phi)_{(i,j-1)}\big)(\alpha_{i,1}\otimes \mathrm{id}_A^{\otimes j-1})\\[3pt]
&&+\big(\mathrm{id}_A\otimes \mathrm{adj}(\Phi)_{(i-1,j)}\big)\\[3pt]
&&+\big(\mathrm{adj}(\bigtriangleup)\otimes
\mathrm{adj}(\Phi)_{(i-1,j-1)}\big)(\mathrm{id}_A\otimes\beta_{i-1,1}\otimes
\mathrm{id}_A^{\otimes j-1}).\end{eqnarray*}This shows that the
map $\mathrm{adj}(\Phi)$ shares the same inductive formula with
the quantum quasi-shuffle product built on $(A,
\mathrm{adj}(\bigtriangleup),\tau)$. Hence we have the conclusion.
\end{proof}

\section{Applications}
In the past decade, the connection between Rota-Baxter algebras
and quasi-shuffles was discovered and well studied (see \cite{GK},
\cite{EG} and the references therein). Quasi-shuffle algebras
built on commutative algebras provide free objects in the category
of commutative Rota-Baxter algebras. Later, in \cite{Lod}, it is
showed that quasi-shuffle algebras built on polynomial algebras
are free in the category of commutative tridendriform algebras.
Motivated by these works, here, we use quantum quasi-shuffle
algebras to provide examples of Rota-Baxter algebras and
tridendriform algebras.

We first recall the definition of Rota-Baxter algebras. For more
information, one can see \cite{GK}.

\begin{definition}Let $\lambda$ be an element in $\mathbb{K}$. A pair $(R,P)$ is called a \emph{Rota-Baxter algebra of weight $\lambda$} if $R$ is a $\mathbb{K}$-algebra and $P$ is a linear endomorphism of $R$ satisfying that for any $x,y\in R$, $$P(x)P(y)=P(xP(y))+P(P(x)y)+\lambda P(xy).$$\end{definition}

Let $(A, m,1_A,\sigma)$ be a unital braided algebra. Then
apparently $(A, \lambda\cdot m,\sigma)$ is a braided algebra for
any $\lambda\in \mathbb{K}$. We denote by $\Join_{\sigma,\lambda}$
the quantum quasi-shuffle product with respect to $(A,
\lambda\cdot m,\sigma)$. By Lemma 3 in \cite{Ba}, the space
$A\otimes T_{\sigma}^{qsh}(A)$ is a unital associative algebra
with the product
$$\lozenge_{\sigma,\lambda}=(m\otimes
\Join_{\sigma,\lambda})(\mathrm{id}_A\otimes \beta \otimes
\mathrm{id}_{T(A)}).$$ We denote by
$\mathcal{R}_{\sigma,\lambda}(A)$ the pair $(A\otimes
T_{\sigma}^{qsh}(A),\lozenge_{\sigma,\lambda})$. We can view
$\mathcal{R}_{\sigma,\lambda}(A)$ as $T^+(A)=\oplus_{i\geq
1}A^{\otimes i}$ at the level of vector spaces. Recently, we have
two products $\Join_{\sigma,\lambda}$ and
$\lozenge_{\sigma,\lambda}$ on $T^+(A)$. This is an example of
2-braided algebras which produce quantum multi-brace algebras (for
the definitions, see \cite{JR}). We define an endomorphism
$P_A:\mathcal{R}_{\sigma,\lambda}(A)\rightarrow\mathcal{R}_{\sigma,\lambda}(A)$
by
\begin{eqnarray*}P_A(a_0\underline{\otimes}u)&=&1_A\underline{\otimes}a_0\otimes u,\ \mathrm{if}\ u\in T^+(A),\\[3pt]
P_A(a_0\underline{\otimes}\nu)&=&1_A\underline{\otimes}\nu\cdot
a_0,\ \mathrm{if}\ \nu \in \mathbb{K} .\end{eqnarray*}

\begin{theorem}Under the assumptions above, the pair $(\mathcal{R}_{\sigma,\lambda}(A), P_A)$ is a Rota-Baxter algebra of weight $\lambda$.\end{theorem}
\begin{proof}Observe that $\beta(u\underline{\otimes}1_A)=1_A\underline{\otimes}u$ for any $u\in T(A)$. Therefore for any $a, b\in A$ and $x\in A^{\otimes i}$,$ y\in A^{\otimes j}$, we have
\begin{eqnarray*}\lefteqn{P\big((a\underline{\otimes}x)\lozenge_{\sigma,\lambda}P(b\underline{\otimes}y)\big)}\\[3pt]
&=&P\big((a\underline{\otimes}x)\lozenge_{\sigma,\lambda}(1_A\underline{\otimes}b\otimes
y)\big)=P\Big(a\underline{\otimes}\big(x\Join_{\sigma,\lambda}(b\otimes
y)\big)\Big)\\[3pt]
&=&1_A\underline{\otimes}a\otimes
\big(x\Join_{\sigma,\lambda}(b\otimes y)\big),\end{eqnarray*}
\begin{eqnarray*}\lefteqn{P\big(P(a\underline{\otimes}x)\lozenge_{\sigma,\lambda}(b\underline{\otimes}y)\big)}\\[3pt]
&=&P\big((1_A\underline{\otimes}a\otimes  x)\lozenge_{\sigma,\lambda}(b\underline{\otimes}y)\big)\\[3pt]
&=&1_A\underline{\otimes}\big((\mathrm{id}_A\otimes
\Join_{\sigma,\lambda,(i+1,j)})(\beta_{i+1,1}\otimes
\mathrm{id}_A^{\otimes j})(a\otimes  x\otimes b\otimes
y)\big),\end{eqnarray*}and
\begin{eqnarray*}\lefteqn{\lambda P\big((a\underline{\otimes}x)\lozenge_{\sigma,\lambda}(b\underline{\otimes}y)\big)}\\[3pt]
&=& 1_A\underline{\otimes}\big((\lambda\cdot m\otimes
\Join_{\sigma,\lambda})(\mathrm{id}_A\otimes \beta_{i1} \otimes
\mathrm{id}_A^{\otimes j})(a\otimes x\otimes b\otimes y)\big)
.\end{eqnarray*} By taking a summation, we get
\begin{eqnarray*}\lefteqn{P\big((a\underline{\otimes}x)\lozenge_{\sigma,\lambda}P(b\underline{\otimes}y)\big)+P\big(P(a\underline{\otimes}x)\lozenge_{\sigma,\lambda}(b\underline{\otimes}y)\big)+\lambda P\big((a\underline{\otimes}x)\lozenge_{\sigma,\lambda}(b\underline{\otimes}y)\big)}\\[3pt]
&=&1_A\underline{\otimes}\Big(a\otimes
\big(x\Join_{\sigma,\lambda}(b\otimes
y)\big)\\[3pt]
&&\ \ \ \ \ \ \ \ \ +(\mathrm{id}_A\otimes
\Join_{\sigma,\lambda,(i+1,j)})(\beta_{i+1,1}\otimes
\mathrm{id}_A^{\otimes j})(a\otimes  x\otimes b\otimes
y)\\[3pt]
&&\ \ \ \ \ \ \ \ \ +(\lambda\cdot m\otimes
\Join_{\sigma,\lambda})(\mathrm{id}_A\otimes \beta_{i1} \otimes
\mathrm{id}_A^{\otimes j})(a\otimes x\otimes b\otimes y)\Big)\\[3pt]
&=&1_A\underline{\otimes}\big((a\otimes x)\Join_{\sigma,\lambda}(b\otimes y)\big)\\[3pt]
&=&P_A(a\underline{\otimes}x)\lozenge_{\sigma,\lambda}P_A(b\underline{\otimes}y).\end{eqnarray*}\end{proof}

\begin{definition}A triple $(R,P,\sigma)$ is called a \emph{braided Rota-Baxter algebra of weight $\lambda$} if $(R,\sigma)$ is a braided algebra and $P$ is an endomorphism of $R$ such that $(R,P)$ is a Rota-Baxter algebra of weight $\lambda$ and $\sigma(P\otimes P)=(P\otimes P)\sigma$.\end{definition}

\begin{example}Let $(A,m,1_A,\sigma)$ be a unital braided algebra and $(\mathcal{R}_{\sigma,\lambda}(A),P_A)$ be the Rota-Baxter algebra defined before. Then $(\mathcal{R}_{\sigma,\lambda}(A),P_A,\beta)$ is a braided Rota baxter of weight $\lambda$.

Indeed, the only thing we need to verify is that $\beta(P_A\otimes
P_A)=(P_A\otimes P_A)\beta$. For any $a, b\in A$ and $x\in
A^{\otimes i}$,$ y\in A^{\otimes j}$, we have
\begin{eqnarray*}\lefteqn{\beta(P_A\otimes
P_A)\big((a\underline{\otimes}x)\underline{\otimes}(b\underline{\otimes}y)\big)}\\[3pt]
&=&\beta\Big(\big(1_A\underline{\otimes}(a\otimes x)\big)\underline{\otimes}\big(1_A\underline{\otimes}(b\otimes y)\big)\Big)\\[3pt]
&=&1_A\underline{\otimes}(\beta_{1,j+1}\otimes \mathrm{id}_A^{\otimes i+1})\Big(1_A\underline{\otimes}\beta_{i+1,j+1}\big((a\otimes x)\underline{\otimes}(b\otimes y)\big)\Big)\\[3pt]
&=&(P_A\otimes
P_A)\beta\big((a\underline{\otimes}x)\underline{\otimes}(b\underline{\otimes}y)\big).\end{eqnarray*}\end{example}

\begin{proposition}Let $(R,P,\sigma)$ be a braided Rota-Baxter algebra of weight $\lambda$. We define $$x\star_P y=xP(y)+P(x)y+\lambda xy,$$ for any $x,y \in R$. If $\sigma(P\otimes \mathrm{id})=(\mathrm{id}\otimes P)\sigma$ and $\sigma(\mathrm{id}\otimes P)=(P\otimes \mathrm{id})\sigma$, then $(R,\star_P,P,\sigma)$ is again a braided Rota-Baxter algebra of weight $\lambda$.\end{proposition}
\begin{proof}It is well-known that $(R,\star_P,P)$ is a Rota-Baxter algebra of weight $\lambda$.
We denote by $m$ the multiplication of $R$. Then we have
\begin{eqnarray*}\lefteqn{\sigma(\star_P\otimes \mathrm{id})}\\[3pt]
&=&\sigma\big((m\otimes \mathrm{id})( \mathrm{id}\otimes P\otimes \mathrm{id})+(m\otimes \mathrm{id})( P\otimes \mathrm{id}\otimes \mathrm{id})+\lambda m\otimes \mathrm{id}\big)\\[3pt]
&=&(\mathrm{id}\otimes m)\sigma_1\sigma_2( \mathrm{id}\otimes P\otimes \mathrm{id})+(\mathrm{id}\otimes m)\sigma_1\sigma_2( P\otimes \mathrm{id}\otimes \mathrm{id})\\[3pt]
&&+(\mathrm{id}\otimes \lambda m)\sigma_1\sigma_2\\[3pt]
&=&\big((\mathrm{id}\otimes m)( \mathrm{id}\otimes\mathrm{id} \otimes P)+(\mathrm{id}\otimes m)(\mathrm{id} \otimes P\otimes \mathrm{id})+(\mathrm{id}\otimes \lambda m)\big)\sigma_1\sigma_2\\[3pt]
&=&(\mathrm{id}\otimes\star_P
)\sigma_1\sigma_2.\end{eqnarray*}

The another condition $\sigma(\mathrm{id}\otimes\star_P
)=(\star_P\otimes \mathrm{id})\sigma_2\sigma_1$ can be verified
similarly.\end{proof} Given any braided Rota-Baxter algebra
$(R,P,\sigma)$, the above proposition provides another example of
2-braided algebra.

Now we turn to tridendriform algebras which were introduced by
Loday and Ronco (\cite{LR}).

\begin{definition}Let $V$ be a vector space, and $\prec$, $\succ$ and $\cdot$ be three binary operations on $V$. The quadruple $(V,\prec, \succ, \cdot)$ is called a \emph{tridendriform algebra} if the following relations are
satisfied: for any $x,y,z\in V$,
\begin{eqnarray*}(x\prec y)\prec z&=&x\prec(y\ast z),\\[3pt]
(x\succ y)\prec z&=&x\succ (y\prec z),\\[3pt]
(x\ast y)\succ z&=&x\succ(y\succ z),\\[3pt]
(x\succ y)\cdot z&=&x\succ(y \cdot z),\\[3pt]
(x\prec y)\cdot z&=&x\cdot (y\succ z),\\[3pt]
(x\cdot y)\prec z&=&x\cdot (y\prec z),\\[3pt]
(x\cdot y)\cdot z&=&x\cdot (y \cdot z),\end{eqnarray*}where $x\ast
y=x\prec y+x\succ y+ x\cdot y$.\end{definition}

\begin{theorem}Let $(A,m,\sigma)$ be a braided algebra. We define three operations $\cdot$, $\prec$ and $\succ$ on $T(A)$ recursively by: for any $a,b\in A$ and any $x\in A^{\otimes i}$, $x\in A^{\otimes i}$,
\begin{eqnarray*}(a\otimes x)\cdot (b\otimes y)&=&(m\otimes\Join_{\sigma (i,j)} )(\mathrm{id}_A\otimes
\beta_{i,1}\otimes \mathrm{id}_A^{\otimes j-1})(a\otimes x\otimes
b\otimes y),\\[3pt]
(a\otimes x)\prec (b\otimes y)&=&a\otimes \big(x\Join_\sigma (b\otimes y)\big),\\[3pt]
(a\otimes x)\succ (b\otimes y)&=&(\mathrm{id}_A\otimes
\Join_{\sigma (i+1,j)})(\beta_{i+1,1}\otimes
\mathrm{id}_A^{\otimes j})(a\otimes x\otimes b\otimes
y).\end{eqnarray*}Then $(T(A), \prec,\succ,\cdot)$ is a
tridendriform algebra.
\end{theorem}
\begin{proof}All the verifications are direct. We just need that $\Join_\sigma$ is associative and compatible with the braiding $\beta$ in the sense of braided algebras. For instance, we show the third condition. For any $a\in A$ and $x,y,z\in T(A)$,
\begin{eqnarray*}
\lefteqn{(x\Join_\sigma y)\succ (a\otimes
z)}\\[3pt]
&=&(\mathrm{id}_A\otimes \Join_{\sigma})(\beta_{?,1}\otimes
\mathrm{id}_{T(A)})(\Join_\sigma \otimes \mathrm{id}_A\otimes \mathrm{id}_{T(A)})(x \underline{\otimes} y\underline{\otimes} a\otimes z)\\[3pt]
&=&(\mathrm{id}_A\otimes \Join_{\sigma})(\beta_{?,1}(\Join_\sigma
\otimes \mathrm{id}_A)\otimes
\mathrm{id}_{T(A)})(x \underline{\otimes} y\underline{\otimes} a\otimes z)\\[3pt]
&=&(\mathrm{id}_A\otimes
\Join_{\sigma})(\mathrm{id}_A\otimes\Join_\sigma  \otimes
\mathrm{id}_{T(A)})\\[3pt]
&&\circ(\beta_{?,1}\otimes \mathrm{id}_{T(A)}\otimes\mathrm{id}_{T(A)})(\mathrm{id}_{T(A)}\otimes\beta_{?,1}\otimes \mathrm{id}_{T(A)})(x \underline{\otimes} y\underline{\otimes} a\otimes z)\\[3pt]
&=&(\mathrm{id}_A\otimes
\Join_{\sigma})(\mathrm{id}_A\otimes\mathrm{id}_{T(A)}
\otimes\Join_\sigma
)\\[3pt]
&&\circ(\beta_{?,1}\otimes \mathrm{id}_{T(A)}\otimes\mathrm{id}_{T(A)})(\mathrm{id}_{T(A)}\otimes\beta_{?,1}\otimes \mathrm{id}_{T(A)})(x \underline{\otimes} y\underline{\otimes} a\otimes z)\\[3pt]
&=&(\mathrm{id}_A\otimes \Join_{\sigma})(\beta_{?,1}\otimes \mathrm{id}_{T(A)})\\[3pt]
&&\circ(\mathrm{id}_{T(A)}\otimes\mathrm{id}_A\otimes\Join_\sigma )(\mathrm{id}_{T(A)}\otimes\beta_{?,1}\otimes \mathrm{id}_{T(A)})(x \underline{\otimes} y\underline{\otimes} a\otimes z)\\[3pt]
&=&x\succ\big(y\succ  (a\otimes z)\big).
\end{eqnarray*}\end{proof}

\begin{remark}Let $(R,\cdot,P)$ be a Rota-Baxter algebra of weight 1. Define $a\prec b=a\cdot P(b)$ and $a\succ b=P(a)\cdot b$. Then $(R,\prec,\succ,\cdot)$ is a tridendriform algebra (see \cite{Lod}). Using this fact, we can prove the above theorem by a more easier argument: embed $A$ into the unital braided algebra $(\widetilde{A},\widetilde{m},1,\widetilde{\sigma})$, then the tridendriform algebra structure in Theorem 5.8 comes from the Rota-Baxter algebra $(\mathcal{R}_{\widetilde{\sigma},1}(\widetilde{A}),P_{\widetilde{A}})$.\end{remark}

\section*{Acknowledgements}
I learned the original construction of quantum quasi-shuffles from
Professor Marc Rosso in 2008. I am grateful to him for encouraging
me to continue the study of this subject. Finally, I would like to
thank Professor Daniel Sternheimer for his suggestion that
improves the title of this paper. This work was partially
supported by NSF of China (Grant No. U0935003), and the grant for
new teachers from DGUT (Grant No. ZJ100501).

\bibliographystyle{amsplain}

\end{document}